\documentclass[11pt,leqno]{amsart}
\usepackage{amsmath,amsthm,amssymb}
\usepackage{tabularx}

\theoremstyle{plain}
\newtheorem{thm}{Theorem}[section]
\newtheorem{prop}{Proposition}[section]
\newtheorem{lem}{Lemma}[section]
\newtheorem{cor}{Corollary}[section]

\theoremstyle{definition}
\newtheorem{df}{Definition}[section]
\newtheorem{rem}{Remark}[section]

\newcommand{\FF}{\mathbb{F}}
\newcommand{\ZZ}{\mathbb{Z}}
\newcommand{\RR}{\mathbb{R}}
\newcommand{\CC}{\mathbb{C}}

\newcommand{\Z}{\mathbb{Z}}
\newcommand{\NN}{\mathbb{N}}
\newcommand{\C}{\mathbb{C}}
\newcommand{\R}{\mathbb{R}}

\newcommand{\GF}{\mathbb{F}}

\DeclareMathOperator{\supp}{supp}
\DeclareMathOperator{\Aut}{Aut}
\DeclareMathOperator{\wt}{wt}

\DeclareMathOperator{\Harm}{Harm}

\DeclareMathOperator{\tr}{tr}

\newcommand{\1}{{\bf 1}}
\def\h{{\mathfrak{h}}}

\newcommand{\w}{\omega}

\newcommand{\eqa}{\begin{eqnarray}}
\newcommand{\eeqa}{\end{eqnarray}}
\newcommand{\eqn}{\begin{eqnarray*}}
\newcommand{\eeqn}{\end{eqnarray*}}

\newcommand{\ord}{{\rm ord}}

\makeatletter
\@addtoreset{equation}{section}

\makeatother

\title[Design-theoretic analogies]{Design-theoretic analogies between 
codes, lattices, and vertex operator algebras}

\author{Tsuyoshi Miezaki*}
\thanks{*Corresponding author}
\address{Faculty of Education, University of the Ryukyus, Okinawa  
903--0213, Japan}
\email{miezaki@edu.u-ryukyu.ac.jp}

\date{February 27, 2020}



\begin{document}

\begin{abstract}
There are many analogies between codes, lattices, and 
vertex operator algebras. 
For example, extremal objects are good examples of 
combinatorial, spherical, and conformal designs. 
In this study, 
we investigated these objects from the aspect of design theory. 
\end{abstract}
\maketitle



\section{Introduction}\label{sec:intro}

In this study, 
we investigated the analogy between 
codes, lattices, and vertex operator algebras (VOAs), with regard to design theory. 
To explain our results, 
we review some of the previous studies conducted on codes, lattices, and VOAs. 

First, we review codes and their combinatorial designs. 
Let $C$ be a doubly even self-dual code of length $n=24m+8r$. 
Then, its minimum weight would satisfy the following equation:
\begin{align}\label{ine:code-ext}
\min(C)\leq 4\left\lfloor\frac{n}{24}\right\rfloor+4. 
\end{align}
We say that $C$ meeting
the bound (\ref{ine:code-ext}) with equality is extremal. 
Let $C$ be an extremal code of length $n=24m+8r$, and 
let us set 
\[
C_\ell:=\{c\in C\mid \wt(c)=\ell\}. 
\]
Then,
any $C_\ell$
forms a combinatorial $t$-design, where
\[
t=
\left\{
\begin{array}{l}
5\ {\rm if}\ n \equiv 0 \pmod{24},\\
3\ {\rm if}\ n \equiv 8 \pmod{24},\\
1\ {\rm if}\ n \equiv 16 \pmod{24}
\end{array}
\right.
\]
\cite{AM}. 

Let $C$ be a doubly even self-dual code of length $n=24m$ with 
$\min(C)=4m$. 
$C$ has the largest minimum weight except for the 
extremal cases. 
Any shell of such code $C$
($C_\ell$) forms a combinatorial $1$-design \cite{MMN}. 

Then, we review lattices and their spherical designs. 
Let $L$ be an even unimodular lattice of rank $n=24m+8r$. 
It was shown in \cite{MOS75} that
its minimum norm satisfies the following equation:
\begin{align}\label{ine:lattice-ext}
\min(L)\leq 2\left\lfloor\frac{n}{24}\right\rfloor+2. 
\end{align}
We say that $L$ meeting
the bound (\ref{ine:lattice-ext}) with equality is extremal. 
Let $L$ be an extremal lattice of rank $n=24m+8r$, and 
let us set 
\[
L_\ell:=\{x\in L\mid (x,x)=\ell\}. 
\]
Then, 
any $L_\ell$
forms a spherical $t$-design, where
\[
t=
\left\{
\begin{array}{l}
11\ {\rm if}\ n \equiv 0 \pmod{24},\\
7\ {\rm if}\ n \equiv 8 \pmod{24},\\
3\ {\rm if}\ n \equiv 16 \pmod{24}
\end{array}
\right.
\]
\cite{Venkov} (see also \cite{Pache}). 

Let $L$ be an even unimodular lattice of rank $n=24m$ with 
$\min(L)=2m$. 
$L$ has the largest minimum norm except for the 
extremal cases. 
It is known that 
any shell of such lattice $L$
($L_\ell$) forms a spherical $3$-design \cite{MMN}. 

Finally, we review VOAs and their conformal designs. 
Let $V$ be a holomorphic VOA of central charge $n=24m+8r$. 
It has been shown in \cite{H} that 
its minimum weight satisfies the following equation:
\begin{align}\label{ine:voa-ext}
\min(V)\leq \left\lfloor\frac{n}{24}\right\rfloor+1. 
\end{align}
We say that $V$ meeting
the bound (\ref{ine:voa-ext}) with equality is extremal. 
Let $V$ be an extremal VOA of central charge $n=24m+8r$. 
$V$ is graded by $L(0)$-eigenvalues as follows: 
\[
V=\bigoplus_{n\in \Z} V_{n}. 
\]
Then, 
any $V_\ell$
forms a conformal $t$-design, where
\[
t=
\left\{
\begin{array}{l}
11\ {\rm if}\ n \equiv 0 \pmod{24},\\
7\ {\rm if}\ n \equiv 8 \pmod{24},\\
3\ {\rm if}\ n \equiv 16 \pmod{24}
\end{array}
\right.
\]
\cite{H1}. 
(For the detailed expressions of 
the minimum weights and the conformal $t$-designs, 
see \cite{H} and \cite{H1}.) 

Let $V$ be a holomorphic VOA of central charge $n=24m$ with 
$\min(V)=m$. 
$V$ has the largest minimum weight except for the 
extremal cases. 
Considering this, the question that arises is 
whether any shell of such VOA $V$
($V_\ell$) forms a conformal $3$-design. 

The first major finding of this study is 
as follows: 
\begin{thm}\label{thm:main1}
Let $V$ be a holomorphic VOA of central charge $n=24m$ with 
$\min(V)=m$. 
Then, any shell of $V$
$($$V_\ell$$)$ forms a conformal $3$-design. 
\end{thm}

The second purpose of this paper is as follows. 
Let $L$ be an even unimodular lattice. 
It is known that $L_{\ell}$ forms a 
spherical $T_2$-design, 
where $T_2$ is the set of positive odd numbers, 
that is, $T_2=\{1,3,5,\cdots\}$. 
(See Section \ref{sec:spherical} for the definition of 
spherical $T$-designs.)
The second major finding is 
as follows: 
\begin{thm}\label{thm:main2}
\begin{enumerate}
\item [{\rm (1)}]
Let $C$ be a doubly even self-dual code of length $n=24m+8r$. 
Then, any $C_\ell\cup C_{n-\ell}$ 
forms a combinatorial $T_2$-design with $2$-weight. 
Further, any $C_\ell\cup C_{n-\ell}$ 
forms a combinatorial $1$-design with $2$-weight. 

\item [{\rm (2)}]
Let $T_2$ be the set of positive odd numbers, 
That is, $T_2=\{1,3,5,\cdots\}$. 
Then, any non-empty homogeneous space of 
a holomorphic VOA forms a conformal $T_2$-design. 

\end{enumerate}
\end{thm}

The third purpose of the present paper is as follows: 
Let $L$ be an even unimodular lattice of rank $8,16,24$.
Then the following holds: 
$L_\ell$ is a spherical $T$-design with 
\[
T=
\begin{cases}
\{1,2,3,4,5,6,7,9,10,11\}\cup T_{2} &if\ \ell=1\\
\{1,2,3,5,6,7\}\cup T_{2} &if\ \ell=2\\
\{1,2,3\}\cup T_{2} &if\ \ell=3.
\end{cases}
\]
The third major finding is 
as follows: 
\begin{thm}\label{thm:main3}
\item
Let $V$ be a 
holomorphic VOA of central charge $c=8\ell$, 
with $\ell=1,2,3$. 
Then, any non-empty 
homogeneous space of $V$ forms a
conformal $T$-design, with 
\[
T=
\begin{cases}
\{1,2,3,4,5,6,7,9,10,11\}\cup T_{2} &if\ \ell=1\\
\{1,2,3,5,6,7\}\cup T_{2} &if\ \ell=2\\
\{1,2,3\}\cup T_{2} &if\ \ell=3.
\end{cases}
\]
All the homogeneous spaces are conformal $3$-designs 
if $c\leq 24$. 
Moreover, all the homogeneous spaces are conformal $7$-designs 
if $c=8$. 
\end{thm}
\begin{rem}
The case $\ell= 1$ and $\ell = 2$ in Theorem \ref{thm:main3} have also been
mentioned in a remark after Theorem 3.1 of \cite{H1}. 
The proof is essentially the same as \cite{H1} 
with a minimal modification. 
\end{rem}

The fourth purpose of this study 
slightly differs from the above three findings. 
A homogeneous space of VOA $V_\ell$ 
has a strength $t$
if $V_\ell$ is a conformal $t$-design 
but is not a conformal $(t+1)$-design. 
We define the concept of strength $t$ for the 
spherical $t$-designs and the combinatorial $t$-designs. 

The fourth purpose of this study is to 
provide other examples for which the strength can be determined 
(Theorem \ref{thm:main4} $(1)$). 

We also present other interesting examples of 
conformal designs. 
All the known examples of conformal designs $V_\ell$ 
have the same strength for each $\ell$. 
This leads to the question of whether 
there are conformal designs $V_\ell$ for which 
the strengths are different for each $\ell$. 
In the final part of this paper, 
we give examples for this (Theorem \ref{thm:main4} $(2)$). 
\begin{thm}\label{thm:main4}
\begin{enumerate}
\item [{\rm (1)}]
Let $L$ be an 
even unimodular lattice of rank $24$. 
Then, all the homogeneous spaces $(V_L)_{\ell}$ have strength $3$. 

\item [{\rm (2)}]
Let $L$ be an 
even unimodular lattice of rank $16$. 
We use ${\rm ord}_{p}(\ell)$ 
to denote the number of times that a prime $p$ occurs 
in the prime factorization of a non-zero integer $\ell$. 
If ${\rm ord}_{p}(3\ell-2)$ is odd for some prime $p\equiv 2\pmod{3}$, 
then all the homogeneous spaces $(V_L)_{\ell}$ have strength $3$. 
Otherwise, the homogeneous spaces $(V_L)_{\ell}$ are conformal $7$-designs.
\end{enumerate}

\end{thm}

\begin{rem}
It is generally difficult 
to determine the strength of $C_\ell$, $L_\ell$, and $V_\ell$ . For example, 
in \cite{{Venkov},{Miezaki}}, the following theorem was shown: 
\begin{thm}[{\cite[Theorem 1.2]{Miezaki}}]\label{thm:Lehmer}
Let $E_8$ be the $E_8$-lattice and 
$V^{\natural}$ be the moonshine VOA. 
Let $\tau(i)$ be Ramanujan's $\tau$-function$:$ 
\begin{eqnarray*}\label{eqn:delta}
\Delta(z)=\eta(z)^{24}=(q^{1/24}\prod_{i\geq 1}(1-q^{i}))^{24}
=\sum_{i\geq 1}\tau (i)q^{i},
\end{eqnarray*}
where $q=e^{2\pi iz}$. Then, the followings are equivalent{\rm :} 
\begin{enumerate}
\item [{\rm (1)}]

$\tau (\ell)=0$. 
\item [{\rm (2)}]

$(E_8)_{2\ell}$ is a spherical $8$-design.
\item [{\rm (3)}]

$(V^{\natural})_{\ell+1}$ is a conformal $12$-design.
\end{enumerate}
\end{thm}
Lehmer's conjecture gives $\tau (i) \neq 0$ \cite{Lehmer}. 
Thus, Theorem \ref{thm:Lehmer} is 
a reformulation of Lehmer's conjecture. 

We have not yet been able to determine 
the strength of $(V^{\natural})_\ell$ for general $\ell$;hence, Lehmer's conjecture is still open. 
This demonstrates the difficulty of determining 
the strength of $V_\ell$ for general $V$.
However, in \cite{{Miezaki},{Miezaki2}}, 
There are examples for which the strength $t$ can be determined. 
It has been shown that the shells of $\ZZ^2$-lattice and
$A_2$-lattice have 
strength $3$ \cite{Miezaki}. 
It has also been shown that the homogeneous spaces in a $d$-free boson VOA have 
strength $3$ \cite{Miezaki}. 
\end{rem}

\begin{rem}
Let $C$ be a doubly even self-dual code of length $n$. 
Let 
\[
\rho:\ZZ^n\rightarrow \FF_2^n;x\mapsto x\pmod 2. 
\]
Then, the construction A of $C$
\[
L_C:=\frac{1}{\sqrt{2}}\{x\in \ZZ^n\mid \rho(x)\in C\}
\]
is an even unimodular lattice. 
Similarly, let $L$ be an even unimodular lattice of rank $n$. 
Then, we obtain a holomorphic VOA $V_L$. 
This leads to the question of whether 
there is any analogy in design 
between ``$C$ and $L_C$" 
and ``$L$ and $V_L$". 

Summalizing our results, 
we have the followings. 
Let $C$ be a doubly even self-dual code of length $24,8,16$.
Then the following holds: 
the $C_\ell$ is a combinatorial $t$-design with 
\[
t=
\left\{
\begin{array}{ll}
3\ &{\rm if}\ n=8,\\
1\ &{\rm if}\ n=16,\\
1\ {\rm or}\ 5\ &{\rm if}\ n = 24.\\
\end{array}
\right.
\]
On the other hand, 
The $(L_C)_\ell$ is a spherical $T$-design with 
\[
T=
\begin{cases}
\{1,2,3,4,5,6,7,9,10,11\}\cup T_{2} &{\rm if}\ n=8,\\
\{1,2,3,5,6,7\}\cup T_{2} &{\rm if}\ n=16,\\
\{1,2,3\}\cup T_{2} &{\rm if}\ n=24.
\end{cases}
\]

Let $L$ be an even unimodular lattice of rank $24,8,16$.
Then, the following holds: 
the $L_\ell$ is a spherical $t$-designs with 
\[
t=
\left\{
\begin{array}{ll}
7\ &{\rm if}\ n=8,\\
3\ &{\rm if}\ n=16,\\
3\ {\rm or}\ 11\ &{\rm if}\ n = 24.\\
\end{array}
\right.
\]
On the other hand, 
The $(V_L)_\ell$ is a spherical $T$-design with 
\[
T=
\begin{cases}
\{1,2,3,4,5,6,7,9,10,11\}\cup T_{2} &{\rm if}\ n=8,\\
\{1,2,3,5,6,7\}\cup T_{2} &{\rm if}\ n=16,\\
\{1,2,3\}\cup T_{2} &{\rm if}\ n=24.
\end{cases}
\]
Therefore, 
there exists an analogy in design 
between ``$C$ and $L_C$" 
and ``$L$ and $V_L$". 
\end{rem}

This paper is organized as follows: 
In Section \ref{sec:pre}, we give 
the definitions of combinatorial, spherical, and 
conformal $t$-designs; 
In Section \ref{sec:main1}, 
we give a proof of Theorem \ref{thm:main1}; 
In Section \ref{sec:main2}, 
we give a proof of Theorem \ref{thm:main2}; 
In Section \ref{sec:main34}, 
we give a proof of Theorem \ref{thm:main3} and \ref{thm:main4}; 
Finally, in Section \ref{sec:rem}, 
we provide some concluding remarks.

\section{Preliminaries}\label{sec:pre}

\subsection{Codes and conbinatorial $t$-designs}\label{sec:conformal}
Let $C$ be a subspace of $\FF_2^n$, 
where $\FF_2$ is the binary finite field. 
$C$ is called a 
$($binary$)$ linear code of length $n$. 
For $x = (x_1,\ldots, x_n) \in \FF_2^n$, 
we put 
\[
\wt(x) = \sharp\{i\mid  x_i = 1\}.
\]
The minimum weight of non-zero elements of $C$ is 
called the minimum weight $\min(C)$
of $C$. 
In this section, we set $(x, y) =
\sum_{i=1}^n x_iy_i$ throughout, for $x = (x_1,\ldots , x_n),
y = (y_1,\ldots , y_n) \in \FF_2^n$.

Let $C$ be a linear code. 
We say $C$
is a doubly even self-dual code 
if 
it is 
doubly even (i.e., $\wt(x) \in 4\ZZ$
for all $x \in C$) 
and 
is self-dual 
(i.e., $C =
C^\perp := \{x \in \FF_2^n
\mid (x, y) = 0 \mbox{ for all } y \in C\}$). 
It is well known that 
if there exists a doubly even self-dual code, 
then $n$ must be a multiple of $8$. 
\cite{{SPLAG},{RS-Handbook}} provides 
the definition of and basic information about 
codes. 

We review the concept of combinatorial $t$-design.
\begin{df}
Let $\Omega=\{1, 2,\ldots,v\}$ be a finite set, $\Omega^{\{k\}}$ be the set of all $k$-element subsets of $\Omega$, 
and $X$ be a subset of $\Omega^{\{k\}}$. 
We say $X$ is a combinatorial $t$-design or $t$-$(v, k, \lambda)$ 
design if, for any $T\in \Omega^{\{t\}}$, 
\[
\sharp \{W\in X\mid  T \subset W\} = \lambda. 
\]
\end{df}

We consider the idea of a combinatorial $t$-design with $2$-weight.
\begin{df}
Let $X$ be a subset of $\Omega^{\{k\}}\cup \Omega^{\{\ell\}}\ (k\neq \ell)$. 
We say $X$ is a combinatorial $t$-design with $2$-weight
or a $t$-$(v, k, \lambda)$ 
design with $2$-weight if, 
for any $T\in \Omega^{\{t\}}$, 
\[
\sharp \{W\in X\mid  T \subset W\} = \lambda. 
\]
\end{df}

Codes provide examples of combinatorial designs and 
combinatorial designs with $2$-weight. 
The support of a non-zero vector 
${\bf x}:=(x_{1}, \dots, x_{n})$, 
$x_{i} \in \GF_{q} = \{ 0,1, \dots, q-1 \}$ is 
the set of indices of its non-zero coordinates: 
${\rm supp} ({\bf x}) = \{ i \mid x_{i} \neq 0 \}$\index{$supp (x)$}. 
Let $X:=\{1,\ldots,n\}$ and 
$\mathcal{B}(C_\ell):=\{\supp({\bf x})\mid {\bf x}\in C_\ell\}$. 
Then, for a code $C$ of length $n$, 
we say that $C_\ell$ is a combinatorial $t$-design (with $2$-weight) if 
$(X,\mathcal{B}(C_\ell))$ is a combinatorial $t$-design (with $2$-weight). 

\subsection{Harmonic weight enumerators}

Here, we discuss some definitions and properties of discrete harmonic functions 
and harmonic weight enumerators \cite{{Delsarte},{Bachoc}}.
Let $\Omega=\{1, 2,\ldots,n\}$ be a finite set 
(which will be the set of coordinates of the code),  
$\widetilde{\Omega}$ be the set of its subsets, and
for all $k= 0,1, \ldots, n$, let 
$\Omega^{\{k\}}$ be the set of its $k$-subsets.
We denote the free real vector 
spaces by $\R \widetilde{\Omega}$, $\R \Omega^{\{k\}}$, 
spanned by the elements of $\widetilde{\Omega}$, $\Omega^{\{k\}}$, 
respectively. 
An element of $\R \Omega^{\{k\}}$ is denoted by
$$f=\sum_{z\in \Omega^{\{k\}}}f(z)z$$
and is identified with the real-valued function on $\Omega^{\{k\}}$ given by 
$z \mapsto f(z)$. 
Such an element $f\in \R \Omega^{\{k\}}$ can be extended to 
 $\widetilde{f}\in \R \widetilde{\Omega}$ by setting, for all $u \in \widetilde{\Omega}$,
$$\widetilde{f}(u)=\sum_{z\in \Omega^{\{k\}}, z\subset u}f(z).$$
If an element $g \in \R \widetilde{\Omega}$ is equal to some $\tilde{f}$, for $f \in \R \Omega^{\{k\}}$, we say that $g$ has degree $k$. 
The linear differential operator $\gamma$ is defined by 
$$\gamma(z):=\sum_{y\in \Omega^{\{k-1\}},y\subset z}y$$
for all $z\in \Omega^{\{k\}}$ and for all $k=0,1, \ldots n$, and $\Harm_{k}$ is the kernel of $\gamma$:
$$\Harm_k:=\ker(\gamma|_{\R \Omega^{\{k\}}}).$$
The following theorem is known: 
\begin{thm}[\cite{Delsarte}]\label{thm:design}
A set $X \subset \Omega^{\{k\}}$ of blocks is a $t$-design 
if and only if 
\[
\sum_{x\in X}\widetilde{f}(x)=0
\]
for all $f\in \Harm_k$, $1\leq k\leq t$. 
\end{thm}
Here, we refer to the concept of the combinatorial $T$-design and 
define the concept of the combinatorial $T$-design with $2$-weight. 
\begin{df}[\cite{{Delsarte},{Bachoc}}]
Let $X$ be a subset of $\Omega^{\{k\}}$. 
$X$ is a combinatorial $T$-design if the condition 
$\sum_{x\in X}\widetilde{f}(x)=0$ holds 
for all $f\in \Harm_j$, $j\in T$. 

\end{df}
\begin{df}
Let $X$ be a subset of $\Omega^{\{k\}}\cup \Omega^{\{\ell\}}\ (k\neq \ell)$. 
$X$ is a combinatorial $T$-design with $2$-weight if the condition 
$\sum_{x\in X}\widetilde{f}(x)=0$ holds 
for all $f\in \Harm_j$, $j\in T$. 
\end{df}

To show Theorem \ref{thm:main2} (1), 
we review the theory of the harmonic weight enumerator 
developed in \cite{Bachoc}. 
\begin{df}[\cite{Bachoc}]
Let $C$ be a binary code of length $n$, and let $f\in\Harm_{k}$. 
The harmonic weight enumerator associated with $C$ and $f$ is
\[
w_{C,f}(x,y)=\sum_{c\in C}\tilde{f}(c)x^{n-\wt(c)}y^{\wt(c)}. 
\]
\end{df}

\begin{lem}[\cite{{Bachoc}}]\label{lem:lempache}
Let $C$ be a doubly even self-dual code. 
Then, for $m>0$, the non-empty shell $C_{m}$ 
is a combinatorial $t$-design if and only if 
\[
a^{f}_{m}=0\ for\ every f\in {\rm Harm}_j,\ 1\leq j\leq t
\]
where $a^{f}_{m}$ is the 
coefficient of the harmonic theta series 
\[
w_{C,f}(x,y)=\sum_{m=0}^{n}a^{f}_{m}x^{n-m}y^{m}. 
\]
\end{lem}

Let 
\begin{align*}
\left\{
\begin{array}{ll}
P_{8}(x,y)=x^8+14x^4y^4+y^8,\\
P_{12}(x,y)=x^2y^2(x^4-y^4)^2,\\ 
P_{18}(x,y)=xy(x^8-y^8)(x^8-34x^4y^4+y^8), \\
P_{24}(x,y)=x^4y^4(x^4-y^4)^4,\\
P_{30}(x,y)=P_{12}(x,y)P_{18}(x,y). 
\end{array} 
\right. 
\end{align*}
We set 
\begin{align*}
I_{G,\chi_{k}}=
\left\{
\begin{array}{ll}
\langle P_8(x,y),P_{24}(x,y)\rangle &\mbox{ if }k \equiv 0\pmod{4}\\
P_{12}(x,y)\langle P_8(x,y),P_{24}(x,y)\rangle &\mbox{ if }k \equiv 2\pmod{4}\\
P_{18}(x,y)\langle P_8(x,y),P_{24}(x,y)\rangle &\mbox{ if }k \equiv 3\pmod{4}\\
P_{30}(x,y)\langle P_8(x,y),P_{24}(x,y)\rangle &\mbox{ if }k \equiv 1\pmod{4}. 
\end{array} 
\right. 
\end{align*}
The space, which includes $w_{C,f}(x,y)$, is 
characterized in \cite{Bachoc}: 
\begin{thm}[\cite{Bachoc}]\label{thm:invariant}
Let $C$ be a doubly even self-dual code of length $n$, 
and let $f \in \Harm_{k}$. 
Then, we have $w_{C,f}(x,y) =(xy)^{k} Z_{C,f} (x,y)$. 
Moreover, the polynomial $Z_{C,f} (x,y)$ is of degree $n-2k$ 
and is in $I_{G, \chi_{k}}$. 
\end{thm}

\subsection{Lattices and spherical $t$-designs}\label{sec:spherical}
The Euclidean lattice provides an example of a spherical design. 
A lattice $L \subset \RR^n$
of dimension $n$
is unimodular if
$L = L^{\sharp}$, where
the dual lattice $L^{\sharp}$ of $L$ is defined as
$\{ x \in {\RR}^n \mid (x,y) \in \ZZ \text{ for all }
y \in L\}$ under the standard inner product $(x,y)$.
The norm of a vector $x$ is defined as $(x,x)$.
The minimum norm $\min(L)$ of a unimodular
lattice $L$ is the smallest norm among all non-zero vectors of $L$.

A unimodular lattice with even norms is said to be even. 
An even unimodular lattice of dimension $n$
exists if and only
if $n \equiv 0 \pmod 8$.

The concept of a spherical $t$-design has been explained by Delsarte et al. 
\cite{DGS}. 
\begin{df}[\cite{DGS}]
Let 
\[
S^{n-1}(r) = \{x = (x_1,\ldots , x_n) \in 
\R ^{n}\ |\ x_1^{2}+ \cdots + x_n^{2} = r^2\}. 
\]
For a positive integer $t$, a finite non-empty set X in the unit sphere
$S^{n-1}(1)$ 
is called a spherical $t$-design in $S^{n-1}(1)$ if the following condition is satisfied:
\[
\frac{1}{|X|}\sum_{x\in X}f(x)=\frac{1}{|S^{n-1}(1)|}\int_{S^{n-1}(1)}f(x)d\sigma (x) 
\]
for all polynomials $f(x) = f(x_1, \ldots ,x_n)$ of degree not exceeding $t$. 
\end{df}
Here, the right-hand side of the equation is the surface integral over the sphere, and $|S^{n-1}(1)|$ denotes the area of the sphere $S^{n-1}(1)$. 
A finite subset $X$ in $S^{n-1}(r)$
is also called a spherical $t$-design if $(1/r)X$ is 
a spherical $t$-design on the unit sphere $S^{n-1}(1)$.
If $X$ is a spherical $t$-design but not a spherical $(t+1)$-design, 
we can say that $X$ has strength $t$. 

Lattices provide examples of spherical $t$-designs. 
We say that $L_\ell$ is a spherical $t$-design 
if $(1/\sqrt{\ell})L_\ell$ is a spherical $t$-design. 

\subsection{Spherical theta series}

We denote by ${\rm {\rm Harm}}_{j}(\R^{n})$ as the set of homogeneous 
harmonic polynomials of degree $j$ on $\R^{n}$. 
The following theorem is known: 
\begin{thm}[\cite{DGS}]\label{sec:DGS}
$X(\subset S^{n-1}(1))$ is a spherical $t$-design if and only if 
the condition 
$\sum_{x\in X}P(x)=0$ holds for all 
$P(x) \in {\rm Harm}_j(\R^n)$ with $1 \leq  j \leq  t$. 
If $X$ is antipodal $(i.e.$, $x\in X\Rightarrow -x\in X$$)$, then 
$X$ is a spherical $t$-design 
if and only if the condition 
\[
\sum_{x\in L_{2m}}P(x)=0 
\]
holds for all $P \in {\rm Harm}_{2j}(\R^n)$ with $1 \leq  2j \leq  t$. 
\end{thm}
Let $T$ be a subset of the natural numbers $\NN=\{1,2,\ldots\}$.
Then, we define the concept of spherical $T$-design as follows:
\begin{df}[\cite{Miezaki2}]
$X$ is a spherical $T$-design if the condition 
\[
\sum_{x\in X}P(x)=0 
\]
holds for all $P(x) \in {\rm Harm}_{j}(\R^n)$ with $j\in T$. 

\end{df}
\begin{rem}
We remark that a spherical $t$-design is actually 
a spherical $\{1,2,\ldots,t\}$-design. 
Therefore, the concept of a spherical $T$-design generalizes that of a spherical $t$-design. 
\end{rem}

Let $\mathbb{H} :=\{z\in\C\mid {\rm Im}(z) >0\}$ be the upper half-plane. 
\begin{df}
Let $L$ be the lattice of $\R^{n}$. Then, for a polynomial $P$, the function 
\[
\vartheta _{L, P} (z):=\sum_{x\in L}P(x)e^{i\pi z(x,x)}
\]
is called the theta series of $L $ weighted by $P$. 
\end{df}

\begin{lem}[\cite{{Venkov},{Venkov2},{Pache}}]\label{lem:lempache}
Let $L$ be an integral lattice in $\R^{n}$. Then, for $m>0$, the non-empty shell $L_{m}$ is a spherical $t$-design if and only if 
\[
a^{(P)}_{m}=0\ for\ every\ P\in {\rm Harm}_{2j}(\R^{n}), \ 1\leq 2j\leq t, 
\]
where $a^{(P)}_{m}$ are the Fourier coefficients of the weighted theta-series 
\[
\vartheta _{L , P}(z)=\sum_{m\geq 0}a^{(P)}_{m}q^{m}, 
\]
where $q=e^{\pi iz}$. 
\end{lem}



For example, we consider an even unimodular lattice $L$. Then, the theta series of $L $ weighted by the harmonic polynomial $P$, $\vartheta_{L, P}(z)$, is in a modular form with respect to $SL_{2}(\Z)$. 
In general, we have the following: 
\begin{lem}[\cite{Pache}]\label{rem:pache_2.1}
Let $L\subset \RR^n$ be an even unimodular lattice of of 
rank $n = 8N$ and
of minimum norm $2M$.
Let 
\begin{align*}
E_4(z)&:=1+240\sum_{n=1}^{\infty}\sigma_{3}(n)q^{n}, \\
E_6(z)&:=1-504\sum_{n=1}^{\infty}\sigma_{5}(n)q^{n}, \\
\Delta(z)&:=\frac{E_4(z)^3-E_6(z)^2}{1728}\\
&=q-24q^2+\cdots, 
\end{align*}
where $\sigma_{k-1}(n):=\sum_{d\mid n}d^{k-1}$. 
Then we have 
for $P\in\Harm_{2j} (\RR_n)$, 
\[
\vartheta_{L,P}\in
\begin{cases}
\CC[E_4,\Delta]& \text{if $j$ is even},\\
E_6\CC[E_4,\Delta]& \text{if $j$ is odd}. 
\end{cases}
\]
More precisely, 
there exist $c_i\in \CC$ such that 
\[
\vartheta_{L,P}=
\begin{cases}
\displaystyle\sum_{i=M}^{[(N+j/2)/3]}
c_i\Delta^i
E_4^{N+j/2-3i}& \text{if $j$ is even},\\
\displaystyle
\sum_{i=M}^{[(N+j/2)/3]}
c_iE_6\Delta^i
E_4^{N+(j-3)/2-3i}&\text{if $j$ is odd}. 
\end{cases}
\]
In particular, 
$\vartheta_{L,P} = 0$
if $j$ is even and $3M > N + j/2$, 
or $j$ is odd and $3M > N + (j-3)/2$. 
\end{lem}
\subsection{VOAs and conformal $t$-designs}\label{sec:conformal}

First, we review some information about VOAs 
that will be presented later 
in this paper. 
See \cite{Bo}, \cite{FHL}, and \cite{FLM} for 
definitions and elementary information about VOAs 
and their modules. 

A VOA $V$ over the field $\CC$ of 
complex numbers is a complex vector space 
equipped with a linear map $Y : V \rightarrow {\rm End}(V)[[z, z^{-1}]]$ 
and two non-zero vectors $\1$ and $\omega$ in $V$, 
satisfying certain axioms (cf. \cite{{FHL},{FLM}}). 
We denote a VOA $V$ by $(V,Y,\1,\w)$. 
For $v \in V$, we write 
\[
Y(v,z) =\sum_{n\in\ZZ}v(n)z^{-n-1}. 
\]
In particular, for $\w \in V$, we write 
\[
Y(\w,z) =\sum_{n\in\ZZ}L(n)z^{-n-2}, 
\]
and $V$ is graded by $L(0)$-eigenvalues: 
$V=\oplus_{n\in \Z} V_{n}$. 
We note that $\{L(n)\mid n \in Z\} \cup \{id_V \}$
forms a Virasoro algebra. 
%
For $V_n$, $n$ is called the weight. 
In this study, 
we assume that 
$V_{n}=0$ for $n<0$, 
and $V_{0}=\C \bold{1}$. For $v \in V_{n}$, 
the operator $v(n-1)$ is homogeneous and is of degree $0$. 
We set $o(v) = v(n-1)$. 
We also assume that the VOA $V$ 
is isomorphic to a direct sum of the 
highest weight modules for the Virasoro algebra, 
i.e., 
\begin{align}\label{eqn:decom}
V=\bigoplus_{n\geq 0}V(n), 
\end{align}
where each $V(n)$ is a sum of the 
highest weight $V_\w$ modules of the 
highest weight $n$ and $V(0)=V_\w$.


In particular, the decomposition (\ref{eqn:decom}) yields
the natural projection map 
\begin{align*}
\pi : V \rightarrow V_{\omega} 
\end{align*}
with the kernel $\oplus_{n>0}V(n)$. 
Next, we give the definition of a conformal $t$-design, which 
is based on Matsuo's study \cite{Matsuo}. 
\begin{df}[\cite{H1}]\label{df:con}
Let $V$ be a VOA of central charge $c$, 
and let $X$ be an $h$-weight subspace of a module of $V$. 
For a positive integer $t$, $X$ is referred to as a conformal $t$-design 
if, for all $v \in V_{n}$ (where $0 \leq n \leq t$), we have
\begin{align*}
{\rm tr}\vert _{X} o(v) = {\rm tr}\vert _{X} o(\pi (v)). 
\end{align*}
\end{df}

Then, it is easy to prove the following theorem: 
\begin{thm}[\cite{H1}]\label{thm:2.3}
Let $X$ be the homogeneous subspace of a module of 
a VOA $V$.
$X$ is a conformal $t$-design if and only if 
the condition 
$\tr|_{X}o(v) = 0$ holds for all 
homogeneous $v \in \ker \pi 
=\bigoplus_{n>0}V(n)$
of weight $n \leq t$.
%
%
%
\end{thm}

\begin{thm}[\cite{H1}]\label{thm:2.4}
Let $V$ be a VOA and 
let $N$ be a $V$-module graded by $\ZZ + h$. 
The following conditions are equivalent:
\begin{enumerate}
\item [{\rm (1)}]
The homogeneous subspaces $N_n$ of $N$ are conformal $t$-designs 
based on $V$ for $n \leq h$.

\item [{\rm (2)}]
For all Virasoro highest weight vectors 
$v \in V_s$ with $0 < s \leq t$ 
and all $n \leq h$ we obtain
\[
\tr|_{N_n} o(v) = 0.
\]
\end{enumerate}
\end{thm}

Let $T$ be a subset of the natural numbers.
As an analogue of the concept of spherical $T$-designs, 
we define the concept of a conformal $T$-design as follows:
\begin{df}
$X$ is a conformal $T$-design if the condition 
$\tr|_{X}o(v) = 0$ holds for all 
homogeneous $v \in \ker \pi 
=\bigoplus_{n>0}V(n)$
of weight $j \in T$.
\end{df}
\begin{rem}
We remark that a conformal $t$-design is actually 
a conformal $\{1,2,\ldots,t\}$-design. 
Therefore, the concept of a conformal $T$-design 
generalizes that of a conformal $t$-design. 
\end{rem}

$V_m$ can be considered to have large symmetry if a homogeneous space of VOA $V_m$
is a conformal $t$-design for higher $t$ \cite{Matsuo}. 
A conformal $t$-design is also a conformal $s$-design 
for all integers $1 \leq s \leq t$. 
Therefore, it is of interest 
to investigate the conformal $t$-design for higher $t$. 


For the notion of admissible, 
we refer to \cite{DLM1988}.
A VOA
is called rational if every admissible module is completely reducible.
A rational VOA $V$ is called 
holomorphic
if the only irreducible module of $V$ up to isomorphism is $V$ itself.
The smallest $h>0$ for which $V(h)\neq 0$ is called the
minimal weight of $V$ and is denoted by $\mu(V)$. 

A holomorphic VOA of central charge $c$ 
exists if and only
if $c \equiv 0 \pmod 8$ \cite{H}. 

\subsection{Graded traces}\label{sec:gra}
In this section, we review the concept of the graded trace. 
As stated earlier, $V$ is a VOA with standard $L(0)$-grading 
\[
V=\bigoplus_{n\geq 0}V_{n}. 
\]
Then, for $v\in V_k$, we define 
the graded trace $Z_{V}(v,z)$ as follows: 
\[
Z_{V}(v, z) = \tr\vert_{V} o(v)q^{L(0)-c/24} = 
q^{-c/24}\sum_{n=0}^{\infty}(\tr\vert_{V_{n}}o(v))q^{n}, 
\]
where $c$ is the central charge of $V$. 
If $v=\1$, then
\[
Z_{V}(\1,z) = \tr\vert_{V} q^{L(0)-c/24} 
= q^{-c/24}\sum_{n=0}^{\infty}(\dim V_{n})q^{n}. 
\]

\begin{thm}[\cite{{Zhu}}]\label{thm:Zhu}
Let $V$ be a holomorphic VOA of central charge $c$. 
Let $v\in V_s$ be a Virasoro highest weight vector of conformal weight $s$. 
Let 
\[
S=
\pm
\begin{pmatrix}
0&1\\
1&0
\end{pmatrix}
\ and \ 
T=
\pm
\begin{pmatrix}
1&1\\
0&1
\end{pmatrix}.
\]
Then 
\[
Z_V(v,q)=q^{-c/24}\sum_{n=0}^{\infty}\tr\vert_{V_n}o(v)q^n
\]
is a meromorphic modular form of weight $s$ for 
$PSL_2(\ZZ)$ with character $\rho$ 
\[
\rho(S)=1\ and\ \rho(T)=e^{-2\pi i c/24}. 
\]
\end{thm}
\begin{thm}[\cite{{DMN}}]\label{thm:DMN1}
Let $L$ be a even unimodular lattice of rank $n$ 
Then, for every element $v$ in $V_L$, we
have
\[
Z(v, z) =\frac{f(v, z)}
{\eta(z)^n},
\]
where $f(v, z)$ is  a sum of modular form of $SL(2,\ZZ)$. 
\end{thm}
\begin{thm}[\cite{{DMN}}]\label{thm:DMN}
Let $P$ be a homogeneous spherical harmonic polynomial
and let $V_L$ be the lattice vertex operator algebra 
associated with an even integral
lattice $L$ of rank $k$. 
Then, there exists a Virasoro highest weight vector $v_P$ 
with the property
\[
Z_{V_L}(v_P,q)=\vartheta_{L,P}(z)/\eta(z)^k,
\]
where 
\[
\eta(z):=q^{1/24}\prod_{i=1}^{\infty}(1-q^i). 
\]
\end{thm}

\subsection{Graded traces for lattice vertex operator algebras}\label{sec:Lehmer}
In this section, to prove the theorem \ref{thm:main3}, 
we investigate the graded trace of lattice VOAs. 

Let $L$ be an even unimodular lattice of rank $n=24m+8r$. 
Then, $V_L$ is a holomorphic VOA of central charge $c=24m+8r$. 
Let $v \in V_{\ell}$ be a Virasoro highest weight vector 
of weight $\ell$. 
It follows from Theorem \ref{thm:Zhu} and \ref{thm:DMN1} that 
$\eta(z)^c Z_{V_L} (v, z)$ is a modular form of weight 
$c/2+\ell=12m+4r+\ell$ for 
$SL_2(\ZZ)$. 


Let $c=8$, and let $v \in (V_{L})_{8}$ 
Then 
\[
\eta(z)^c Z_{V_L} (v, z)=c_1(v)(q+\cdots)
\]
is a modular form of weight $12$. 
Therefore, we have 
\[
\eta(z)^c Z_{V_L} (v, z)=c_1(v)\Delta(z)
\]
and
\begin{align}\label{df:a(m)}
Z_{V_L}(v,z)&=\frac{c_1(v)\Delta(z)}{\eta(z)^8}=c_{1}(v)\eta(z)^{16}
=c_{1}(v)q^{-1/3}\sum_{i=1}^{\infty}a(i)q^{i}\ (\mbox{say}), 
\end{align}
where $c_{1}(v)$ is a constant that depends on $v$. 
Let $c=16$ and $v \in (V_L)_{4}$. 
Then 
\[
\eta(z)^c Z_{V_L} (v, z)=c_2(v)(q+\cdots)
\]
is a modular form of weight $12$. 
Therefore, we have 
\[
\eta(z)^c Z_{V_L} (v, z)=c_2(v)\Delta(z)
\]
and
\begin{align}\label{df:b(m)}
Z_{V_L}(v,z)&=\frac{c_{2}(v)\Delta(z)}{\eta(z)^{16}}=c_{2}(v)\eta(z)^{8}
=c_{2}(v)q^{-2/3}\sum_{i=1}^{\infty}b(i)q^{i}\ (\mbox{say}), 
\end{align}
where $c_{2}(v)$ is a constant that depends on $v$. 
Let $c=24$, and let $v \in (V_L)_{4}$ 
be a Virasoro highest weight vector of weight $4$. 
Then 
\[
\eta(z)^c Z_{V_L} (v, z)=c_3(v)(q+\cdots)
\]
is a modular form of weight $16$. 
Therefore, we have 
\[
\eta(z)^c Z_{V_L} (v, z)=c_3(v)E_4(z)\Delta(z)
\]
and
\begin{align}\label{df:c(m)}
Z_{V_L}(v,z)&=\frac{c_3(v)E_4(z)\Delta(z)}{\eta(z)^{24}}=c_{3}(v)E_{4}(z)
=c_{3}(v)q^{-1}\sum_{i=1}^{\infty}c(i)q^{i}\ (\mbox{say}), 
\end{align}
where $c_{3}(v)$ is a constant that depends on $v$. 

Then, using an argument 
similar to that presented in the proof of \cite[Theorem 1.2]{Miezaki}, 
we have the following proposition: 
\begin{prop}\label{prop:refor}
Let the notation be the same as before. 
Then, the following ${\rm (i)}$ and ${\rm (ii)}$ are equivalent for all $c\in \{8,16,24\}$$:$ 
\begin{enumerate}
\item [{\rm (1)}]
Case $c=8${\rm :} 
\begin{enumerate}
\item [{\rm (i)}]
$a(\ell)=0$; 
\item [{\rm (ii)}]
$(V_L)_{\ell}$ is a conformal $8$-design.
\end{enumerate}

\item [{\rm (2)}]
Case $c=16${\rm :} 
\begin{enumerate}
\item [{\rm (i)}]
$b(\ell)=0$; 
\item [{\rm (ii)}]
$(V_L)_{\ell}$ is a conformal $4$-design.
\end{enumerate}

\item [{\rm (3)}]
Case $c=24${\rm :} 
\begin{enumerate}
\item [{\rm (i)}]
$c(\ell)=0$; 
\item [{\rm (ii)}]
$(V_L)_{\ell}$ is a conformal $4$-design.
\end{enumerate}
\end{enumerate}
\end{prop}
\begin{proof}

\begin{enumerate}
\item [(1)]
Let $c=8$. Note that $L\cong E_8$. 
First, note that for $v\in (V_L)_{8}$, by (\ref{df:a(m)}), 
we have 
\[
\eta(z)^{8} Z_{V_{L}}(v,z)=c_1(v)\Delta(z) \in M_{12}(SL_2(\ZZ)).
\]
Assume that $a(\ell)=0$. 
Then, for any $v\in (V_{L})_{8}$, 
we have $\tr\vert_{(V_{L})_{\ell}}o(v)= 0$. 
Therefore, 
$(V_{L})_{\ell}$ is a conformal $8$-design.

Next, we assume the contrary, that is, $a(\ell)\neq 0$. 
Since 
$(V_{E_8})_{1}$ is not a conformal $8$-design 
(cf~\cite[Theorem 4.2 (i)]{H1}), 
by (\ref{df:a(m)}), 
there exists $v\in (V_{E_8})_{8}$ 
of weight $8$ 
such that 
\[
Z_{V_{E_8}}(v,z)
=c_1(v)q^{-1/3}\sum_{i=1}^{\infty}a(i)q^{i}, 
\]
where $c_1(v)\neq 0$. 
Hence, we have 
\[
\tr\vert_{(V_{E_8})_{\ell}}o(v)=c_1(v)\times a(\ell)\neq 0, 
\]
which implies that 
$(V_{E_8})_{\ell}$ is not a conformal $8$-design. 
This completes the proof of Case $1$.


\item [(2)]
Let $c=16$. 
For $v\in (V_L)_{4}$, by (\ref{df:b(m)}), 
we have 
\[
\eta(z)^{16} Z_{V_{L}}(v,z)=c_2(v)\Delta(z) \in M_{12}(SL_2(\ZZ)), 
\]
where $c_2(v)$ is a constant that depends on $v$. 
Assume that $b(\ell)=0$. 
Then, for any $v\in (V_{L})_{4}$, 
we have $\tr\vert_{(V_{L})_{\ell}}o(v)= 0$. 
Therefore, 
$(V_{L})_{\ell}$ is a conformal $4$-design.

On the other hand, 
based on \cite[Lemma 31]{Pache}, 
there exists $P\in {\rm Harm}_4(\RR^{16})$ such that 
$\vartheta_{L,P}(z)=d_1(P)\Delta(z)$, 
where $d_1(P)$ is a non-zero constant. 
Therefore, based on Theorem \ref{thm:DMN}, 
there exists $v_P\in (V_L)_4$ such that 
\[
Z_{V_L}(v_P,z)=d_1(P)\Delta(z)/\eta(z)^{16}
=d_1(P)q^{-2/3}\sum_{i=1}^{\infty}b(i)q^{i}. 
\]
We have $c(P)\times b(1)\neq 0$, that is, 
$(V_{L})_1$ is not a conformal $4$-design. 
Then, the rest of the proof is similar to that of Case $1$. 

\item [(3)]

Let $c=24$. 
For $v\in (V_L)_{4}$, by (\ref{df:c(m)}), 
we have 
\[
\eta(z)^{24} Z_{V_{L}}(v,z)=c_3(v)E_4(z)\Delta(z) \in M_{16}(SL_2(\ZZ)), 
\]
where $c_3(v)$ is a constant that depends on $v$. 
Assume that $c(\ell)=0$. 
Then, for any $v\in (V_{L})_{4}$, 
we have $\tr\vert_{(V_{L})_{\ell}}o(v)= 0$. 
Therefore, 
$(V_{L})_{\ell}$ is a conformal $4$-design.

Let $L$ be a lattice that is not a Leech lattice. Then, 
based on \cite[Lemma 31]{Pache}, 
there exists $P\in {\rm Harm}_4(\RR^{24})$ such that 
$\vartheta_{L,P}(z)= d_2(P)E_4(z)\Delta(z)$, 
where $d_2(P)$ is a non-zero constant. 
Therefore, based on Theorem \ref{thm:DMN}, 
there exists $v_P\in (V_L)_4$ such that 
\[
Z_{V_L}(v_P,z)=d_2(P)E_{4}(z)\Delta(z)/\eta(z)^{24}
=q^{-1}\sum_{i=1}^{\infty}c(i)q^{i}. 
\]
We have $d_2(P)\times c(1)\neq 0$, that is, 
$(V_{L})_1$ is not a conformal $4$-design. 
The rest of the proof is similar to that of Case $1$. 

Let $L$ be a Leech lattice. 
Then, $(V_{L})_{1}=\langle h_1(-1)\1,\ldots,h_{24}(-1)\1
\rangle$, 
where $\{h_{i}\}_{i=1}^{24}$ are the orthonormal basis of $\h$. 
Let $$v_{4}=h_1(-1)^4\1-2h_1(-3)h_1(-1)\1+\frac{3}{2}h_1(-2)^2\1.$$
Then, $v_{4}$ is the highest weight 
vector in $(V_{L})_{4}$ 
(see \cite[Proposition 3.2]{Miezaki}).
Then, we have $\tr\vert_{(V_{L})_{1}}o(v_{4})\neq 0$, 
That is, $(V_{L})_{1}$ is not a conformal $4$-design. 
The rest of the proof is similar to that of Case $1$. 

\end{enumerate}

\end{proof}

\section{Proof of Theorem \ref{thm:main1}}\label{sec:main1}
In this section, we show Theorem \ref{thm:main1}. 
\begin{proof}[Proof of Theorem \ref{thm:main1}]

Let $V$ be a holomorphic VOA of central charge $c=24m$. 
Let $v \in V_{\ell}$ be a Virasoro highest weight vector 
of weight $\ell$. 
It follows from Theorem from Theorem \ref{thm:Zhu} and \ref{thm:DMN1} that 
$\eta(z)^c Z_V (v, z)$ is a modular form of weight 
\[
c/2+\ell=12m+\ell
\]
for 
$SL_2(\ZZ)$. 

Assume that $\ell=1$ or $\ell=3$. 
Then, there is no non-zero 
holomorphic modular form 
of weight $12m+\ell$. 

Assume that $\ell=2$. 
Then 
\[
\eta(z)^{c} Z_V (v, z)=q^{\frac{c}{24}}(1+\cdots)^{\frac{c}{24}}q^{-\frac{c}{24}}(c_mq^m+\cdots)=c_mq^m+\cdots.
\]
Then, there is no non-zero 
holomorphic modular form 
of weight $12m+2$ such that 
the leading term is $c_mq^m+\cdots$, that is, $Z_V (v, q) = 0$. 
Therefore, we can say that 
any homogeneous spaces of $V$ are conformal $3$-design, 
by Theorem \ref{thm:2.4}. 

\end{proof}

\section{Proof of Theorem \ref{thm:main2}}\label{sec:main2}
\subsection{Proof of Theorem \ref{thm:main2} (1)}
In this section, we show Theorem \ref{thm:main2} (1). 
\begin{proof}[Proof of Theorem \ref{thm:main2} (1)]
Let $f \in \Harm_{\ell}$ with $\ell\in T_2$. 
Let \[
w_{C,f}(x,y)=\sum_{i=0}^{n}c_{C,f}(i)x^{n-i}y^i. 
\]
It is sufficient to show that 
$c_{C,f}(\ell)+c_{C,f}(n-\ell)=0$. 
Let $\ell\equiv 3\pmod{4}$. Then, from Theorem \ref{thm:invariant}, 
\[
w_{C,f}(x,y)=P_{18}(x,y)\times (\mbox{a polynomial of } 
P_8(x,y) \mbox{ and } P_{24}(x,y)). 
\]
Note that $P_8(x,y)=P_8(y,x)$, $P_{24}(x,y)=P_{24}(y,x)$ and 
$P_{18}(x,y)=-P_{18}(y,x)$. 
These imply that $w_{C,f}(x,y)=-w_{C,f}(y,x)$ and 
$c_{C,f}(\ell)=-c_{C,f}(n-\ell)$. 
This completes the proof for the case $\ell\equiv 3\pmod{4}$. 
The case $\ell\equiv 1\pmod{4}$ can be proved in a similar manner. 
\end{proof}
The following corollary is 
obtained in \cite{{alltop},{MN}}. 
Theorem \ref{thm:main2} gives a new proof. 
\begin{cor}[\cite{{alltop},{MN}}]\label{cor:main2}
Let $C$ be a doubly even self-dual code of length $n=24m+8r$. 
Then, any $C_{n/2}$ 
forms a combinatorial 
$T_2$-design. 
In particular, any $C_{n/2}$ 
forms a combinatorial 
$1$-design. 
\end{cor}
\begin{proof}
From Theorem \ref{thm:main2}, 
$C_k\cup C_{n-k}$ 
forms a combinatorial 
$T_2$-design with $2$-weight. 
If $k=n/2$, then $C_{n/2}\cup C_{n-n/2}=C_{n/2}$. 
This completes the proof of Corollary \ref{cor:main2}. 
\end{proof}

\subsection{Proof of Theorem \ref{thm:main2} (2)}
In this section, we show Theorem \ref{thm:main2} (2). 
\begin{proof}[Proof of Theorem \ref{thm:main2} (2).] 
Let $V$ be a holomorphic VOA of central charge $c=24n+8r$. 
Let $v \in V_{\ell}$ be a Virasoro highest weight vector 
of weight $\ell$. 
It follows from Theorem \ref{thm:Zhu} and \ref{thm:DMN1} that 
$\eta(z)^c Z_V (v, z)$ is a modular form of weight 
$c/2+\ell=12n+4r+\ell$ for 
$SL_2(\ZZ)$. Assume that $\ell \equiv 1 \pmod 2$. 
Then, there is no non-zero holomorphic modular form of 
odd weight $c/2+\ell$, that is, $Z_V (v, q) = 0$. 
Therefore, we have that 
any homogeneous spaces of $V$ are conformal $T$-design, 
by Theorem \ref{thm:2.4}. 

\end{proof}

\section{Proof of Theorem \ref{thm:main3} and \ref{thm:main4}}\label{sec:main34}
In this section, we give the proof of 
Theorem \ref{thm:main3} and \ref{thm:main4}.

\subsection{Proof of Theorem \ref{thm:main3}}

\begin{proof}[Proof of Theorem \ref{thm:main3}] 

Let $V$ be a holomorphic VOA of central charge $c=8$. 
Let $v \in V_{\ell}$ be a Virasoro highest weight vector 
of weight $\ell$. 
Then, 
$\eta(z)^{c} Z_V (v, z)$ is a modular form of weight 
$4+\ell$ for 
$SL_2(\ZZ)$. 
However, there is no non-zero holomorphic modular form of 
weight $4+\ell$ with 
\[
\ell\in \{1,2,3,4,5,6,7,9,10,11\}, 
\]
that is, $Z_V (v, z) = 0$. 
Therefore, we conclude that 
any homogeneous spaces of $V$ are 
\[
{\rm conformal\ }
\{1,2,3,4,5,6,7,9,10,11\}\cup T_{2}{\rm -design}, 
\]
from Theorem \ref{thm:2.4} and Theorem \ref{thm:main2}. 
In particular, 
any homogeneous spaces of $V$ are conformal 
$7$-designs. 
The proof for the cases $c = 16,24$ are similar. 
This completes the proof of Theorem \ref{thm:main2}. 
\end{proof}

\subsection{Proof of Theorem \ref{thm:main4}}

\begin{proof}[Proof of Theorem \ref{thm:main4} (1)] 
By Proposition \ref{prop:refor}, 
it is sufficient to show that if ${\ord_p (3\ell-2)}$ is odd, then 
$b(\ell)=0$; otherwise $b(\ell)\neq 0$, 
where $b(\ell)$ is defined by (\ref{df:b(m)}). 
Recall that 
\begin{align*}
Z_{V_{L}}(v,z)&=\frac{c_{2}(v)\Delta(z)}{\eta(z)^{16}}=c_{2}(v)\eta(z)^{8}\\
&=c_{2}(v)q^{-2/3}\sum_{i=1}^{\infty}b(i)q^{i} \\
&=c_{2}(v)q^{-16/24}(q-8 q^2+20 q^3-70 q^5+64 q^6+56 q^7-125 q^9+\cdots). 
\end{align*}
Set
\begin{align}\label{eqn:eta3}
\eta(3z)^{8}&=\sum_{i=1}^{\infty}b^{\prime}(i)q^{i}\\
&=q-8q^4+20q^7-70q^{13}+\cdots \nonumber.
\end{align}
The exponents of the power series of (\ref{eqn:eta3}) 
are $1$ modulo $3$. 
By \cite[Theorem 2.1, Corollary 2.2]{HO}, 
for $p\equiv 2\pmod{3}$,
if ${\ord_p (\ell)}$ is odd, then 
$b^{\prime}(\ell)=0$; otherwise $b^{\prime}(\ell)\neq 0$, and 
if $(\ell,n)=1$, then 
\begin{eqnarray*}
b^{\prime}(\ell n)=b^{\prime}(\ell)b^{\prime}(n). \label{eqn:mul}
\end{eqnarray*}
Using these properties of $b^{\prime}(\ell)$, 
if ${\rm ord}_{p}(3\ell-2)$ is odd for some prime $p\equiv 2\pmod{3}$, 
then $(V_{L})_{\ell}$ is a conformal $4$-design; 
otherwise, the homogeneous spaces 
$(V_{L})_{v}$ are not conformal $4$-designs. 

Finally, we show that 
if ${\rm ord}_{p}(3\ell-2)$ is odd for some prime $p\equiv 2\pmod{3}$, 
then $(V_{L})_{\ell}$ is a conformal $7$-design. 
From Theorem \ref{thm:main2}, 
we show that $(V_{L})_{\ell}$ is a 
conformal $\{1,2,3,5,6,7\}\cup T_{2}$-design. 
As shown above, 
for $v\in (V_L)_{k}$ ($5\leq k\leq 7$), 
we conclude that $(V_{L})_{\ell}$ is a conformal $4$-design, 
that is, it is a conformal $\{1,2,3,4,5,6,7\}\cup T_{2}$-design. 
Hence, $(V_{L})_{\ell}$ is a conformal $7$-design.
Thus, the proof is complete. 
\end{proof}

\begin{proof}[Proof of Theorem \ref{thm:main4} (2)] 
Based on Proposition \ref{prop:refor}, 
it is sufficient to show that for $\ell\geq 1$, 
$c(\ell)\neq 0$, where $c(\ell)$ is defined by (\ref{df:c(m)}). 
Based on (\ref{df:c(m)}), we have 
$c(\ell)=\sigma_{3}(\ell)$, 
where $\sigma_{3}(j)$ is a divisor function 
$\sigma_{3}(j)=\sum_{d\vert j}d^{3}$. 
Then, for $\ell\geq 1$, we have $c(\ell)=\sigma_{3}(\ell)\neq 0$. 
\end{proof}


\section{Concluding Remarks}\label{sec:rem}
\begin{enumerate}


\item 
[(1)]
Let $L$ be an even unimodular lattice of rank $16$. 
We showed in Theorem \ref{thm:main4} that, 
if ${\rm ord}_{p}(3\ell-2)$ is odd for some prime $p\equiv 2\pmod{3}$, 
then $(V_{L})_{\ell}$ is a conformal $7$-design. 
It is an interesting, unsolved problem to 
determine whether $(V_{L})_{\ell}$ is 
a conformal $8$-design. 
Let $E_{4}(z)\eta(z)^{8}=
q^{-2/3}\sum_{i=1}^{\infty}d(i)q^i$. 
Using the same method as in the proof of 
Proposition \ref{prop:refor}, 
we can say that $(V_{L})_{\ell}$ is 
a conformal $8$-design if and only if 
$d(\ell)=0$. 

\item 
[(2)]
Let $L$ be an even unimodular lattice of rank $8$ 
(i.e., $L=E_8$-lattice). 
Then, 
based on Proposition \ref{prop:refor}, 
the homogeneous space $(V_{L})_{\ell}$ is 
a conformal $8$-design if and only if 
$a(\ell)=0$, where $a(\ell)$ is defined as follows: $\eta(z)^{16}
=q^{-1/3}\sum_{i=1}^{\infty}a(i)q^{i}$. 
It is conjectured in \cite{Serre} that 
$a(\ell)\neq 0$ for all $\ell$. 
Using the same argument as in \cite{{Lehmer},{BM1},{BM2}}, 
we can say that, 
if $a(p)\neq 0$ for all prime numbers $p$, 
then $a(\ell)\neq 0$ for all $\ell$. 

\item 
[(3)]
Note that there is no known combinatorial $6$-design 
among the $C_\ell$ of code $C$. 
Also note that there are no known spherical or conformal $12$-designs 
among the shell and the homogeneous spaces of any 
lattices or VOAs, except for 
the trivial case $V_{A_{1}}$~\cite[Example 2.6.]{H1}. 
It is an interesting, unsolved
problem to show whether there exists 
a combinatorial $6$-design obtainable from codes, 
a spherical $12$-design obtainable from lattices, or 
a conformal $12$-design obtainable from VOAs. 

\item 
[(4)]
Let $L=A_{1}$-lattice (namely, $L=\sqrt{2}\ZZ=\langle \alpha\rangle_{\ZZ}$). 
Then, all homogeneous spaces of 
the lattice VOA $V_{L}$ 
are conformal $t$-designs for all $t$ (cf.~\cite{H1}). 
This is because $(V_{L})^{\Aut(V_{L})}=V_{\w}$. 
Here, let $\theta$ be an element in $\Aut(V_{L})$ of order $2$, 
which is a lift of $-1\in \Aut(L)$, and 
let $V_{L}^{+}$ be the fixed points of the 
VOA $V_L$ associated with $\theta$. 
Then, all the homogeneous spaces of 
$V_{L}^{+}$ 
are conformal $3$-designs because 
$((V_{L})^{\Aut(V_{L})})_{\leq 3}=(V_{\w})_{\leq 3}$ and 
because of {\cite[Theorem 2.5]{H1}}. 
On the other hand, let 
$v_{4}=\alpha(-1)^4\1-2\alpha(-3)\alpha(-1)\1+\frac{3}{2}\alpha(-2)^2\1
\in (V_{L}^{+})_{4}$. 
Then, we calculate the graded trace as follows \cite{DG}: 
\[
Z_{V_{L}^{+}}(v_{4},z)=q^{1/24}\frac{\eta(2z)^{15}}{\eta(z)^{7}}. 
\]
Therefore, if the Fourier coefficients of $Z_{V_{L}^{+}}(v_{4},z)$ 
do not vanish, then none of the homogeneous spaces of $V_{L}^{+}$ 
are conformal $4$-designs. 
We have checked numerically 
that the coefficients do not vanish up to the exponent $1000$. 

\item 
[(5)]
Using \cite{Pache} and \cite{BM1}, 
we show the following theorem: 
\begin{thm}[cf.~\cite{Pache}, \cite{BM1}]
The shells in the $\mathbb{Z}^{2}$-lattice are spherical $3$-designs 
and are not spherical $4$-designs. 
The shells in the $A_{2}$-lattice are spherical $5$-designs 
and are not spherical $6$-designs. 
\end{thm}
In \cite{MMS}, they showed that 
the homogeneous spaces of $V_{A_{2}}$ are conformal $5$-designs. 
Therefore, it is natural to ask whether 
the corresponding results hold for the lattice VOAs 
$V_{\sqrt{2}\ZZ^{2}}$ and $V_{A_{2}}$. 
Specifically, 
\begin{enumerate}

\item 
Are the homogeneous spaces of $V_{\sqrt{2}\ZZ^{2}}$ conformal $3$-designs 
and not conformal $4$-designs?

\item
Are the homogeneous spaces of $V_{A_{2}}$ 
not conformal $6$-designs?
\end{enumerate}
As interesting, unsolved problems, these questions remain to be answered in the affirmative or negative\footnote{Prof.~Shimakura pointed us 
to the following:
It can be proved that 
the homogeneous spaces of $V_{\sqrt{2}\ZZ^{2}}$ are conformal $3$-designs 
and not conformal $4$-designs. Moreover, 
the homogeneous spaces of $V_{A_{2}}$ are
not conformal $6$-designs \cite{Shimakura}. 
}
\end{enumerate}

\section*{Acknowledgments}
The author would like to thank Prof.~Hiroki Shimakura, 
for his helpful discussions and contributions to this research. 
The author would also like to
thank the anonymous reviewers for their beneficial comments on an earlier
version of the manuscript. 
This work was supported by JSPS KAKENHI (18K03217).

\end{document}